\documentclass[12pt]{amsart}

\usepackage{enumerate}

\usepackage[foot]{amsaddr}
\usepackage[english]{babel}

\allowdisplaybreaks

\usepackage[
pdftex,hyperfootnotes]{hyperref}
\hypersetup{
colorlinks=true,
linkcolor=NavyBlue, 
urlcolor=RoyalPurple,
citecolor=OliveGreen,
pdftitle={Bidiagonal factorization of Hessenberg or Banded matrices},
bookmarks=true,
}
\usepackage[usenames,dvipsnames,svgnames,table,x11names]{xcolor}
\usepackage{pgfplots}
\usepackage{tikz}
\usepackage{tikz-3dplot}
\usetikzlibrary{automata,quotes, chains,matrix,calc,shadows,shapes.callouts,shapes.geometric,shapes.misc,positioning,patterns,decorations.shapes,
decorations.pathmorphing,decorations.markings,decorations.fractals,decorations.pathreplacing,shadings,fadings,arrows.meta,bending}

\usepackage{nicematrix,drawmatrix}
\usepackage[utf8]{inputenc}

\usepackage{comment}

\usepackage
[total={16cm,20cm},
top=4.25cm, 
left=2.5cm,
lmargin=3cm, 
rmargin=3cm]
{geometry}

\usepackage{amssymb,latexsym,amsmath,amsthm,bm}
\usepackage{mathrsfs}
\usepackage{mathtools,arydshln,mathdots}
\mathtoolsset{showonlyrefs}

\usepackage{tcolorbox}

\usepackage[table]{xcolor}

\usepackage{Baskervaldx}
\usepackage[]{newtxmath}

\usepackage{bigints}

\theoremstyle{plain}

\newtheorem{teo}{Theorem}[section]
\newtheorem{coro}[teo]{Corollary}

\newtheorem{pro}[teo]{Proposition}

\renewcommand{\d}{\operatorname{d}}
\newcommand{\Exp}[1]{\operatorname{e}^{#1}}

\allowdisplaybreaks[1]

\makeatletter
\DeclareRobustCommand{\gaussk}{\DOTSB\gaussk@\slimits@}
\newcommand{\gaussk@}{\mathop{\vphantom{\sum}\mathpalette\bigcal@{K}}}

\newcommand{\bigcal@}[2]{%
\vcenter{\m@th
\sbox\z@{\(#1\sum\)}%
\dimen@=\dimexpr\ht\z@+\dp\z@
\hbox{\resizebox{!}{0.8\dimen@}{\( \mathcal{K} \)}}%
}%
}
\newcommand{\cfracplus}{\mathbin{\cfracplus@}}
\newcommand{\cfracplus@}{%
\sbox\z@{\(\dfrac{1}{1}\)}%
\sbox\tw@{\(+\)}%
\raisebox{\dimexpr\dp\tw@-\dp\z@\relax}{\(+\)}%
}
\newcommand{\cfracdots}{\mathord{\cfracdots@}}
\newcommand{\cfracdots@}{%
\sbox\z@{\(\dfrac{1}{1}\)}%
\sbox\tw@{\(+\)}%
\raisebox{\dimexpr\dp\tw@-\dp\z@\relax}{\(\cdots\)}%
}
\makeatother

\makeatletter
\newcommand*{\relrelbarsep}{.386ex}
\newcommand*{\relrelbar}{%
\mathrel{%
\mathpalette\@relrelbar\relrelbarsep
}%
}
\newcommand*{\@relrelbar}[2]{%
\raise#2\hbox to 0pt{\(\m@th#1\relbar\)\hss}%
\lower#2\hbox{\(\m@th#1\relbar\)}%
}
\providecommand*{\rightrightarrowsfill@}{%
\arrowfill@\relrelbar\relrelbar\rightrightarrows
}
\providecommand*{\leftleftarrowsfill@}{%
\arrowfill@\leftleftarrows\relrelbar\relrelbar
}
\providecommand*{\xrightrightarrows}[2][]{%
\ext@arrow 0359\rightrightarrowsfill@{#1}{#2}%
}
\providecommand*{\xleftleftarrows}[2][]{%
\ext@arrow 3095\leftleftarrowsfill@{#1}{#2}%
}
\makeatother

\catcode`,\active

\catcode`\,12

\usepackage{appendix}

\usepackage{xcolor} 

\usepackage[normalem]{ulem}
\usepackage{soul} 

\usepackage{comment} 

\usepackage{tikz}
\usetikzlibrary{shapes,arrows}
\usepackage{verbatim}

\tikzstyle{block} = [draw, rectangle, 
minimum height=3em, minimum width=2em]

\usepackage{mathrsfs} 

\begin{document}
 \title[Bidiagonal factorization of banded matrices]
 {Bidiagonal factorization of recurrence banded matrices in mixed multiple orthogonality}
 
\author[A Branquinho]{Amílcar Branquinho\(^{1}\)}
\address{\(^1\)CMUC, Departamento de Matemática,
 Universidade de Coimbra, 3001-454 Coimbra, Portugal}
\email{\(^1\)ajplb@mat.uc.pt}

\author[JEF Díaz]{Juan EF Díaz\(^{2}\)}
\address{\(^{2,3,4}\)CIDMA, Departamento de Matemática, Universidade de Aveiro, 3810-193 Aveiro, Portugal}
\email{\(^2\)juan.enri@ua.pt}

\author[A Foulquié]{Ana Foulquié-Moreno\(^{3}\)}
\email{\(^3\)foulquie@ua.pt}

\author[H Lima]{Hélder Lima\(^{4}\)}
\email{\(^4\)helder.lima@ua.pt}

\author[M Mañas]{Manuel Mañas\(^{5}\)}
\address{\(^5\)Departamento de Física Teórica, Universidad Complutense de Madrid, Plaza Ciencias 1, 28040-Madrid, Spain 
}
\email{\(^5\)manuel.manas@ucm.es}



\begin{abstract}
This paper demonstrates how to explicitly construct a bidiagonal factorization of the banded recurrence matrix that appears in mixed multiple orthogonality on the step-line in terms of the coefficients of the mixed multiple orthogonal polynomials. The construction is based on the \(LU\) factorization of the moment matrix and Christoffel transformations applied to the matrix of measures and the associated mixed multiple orthogonal polynomials. 
\end{abstract}

\maketitle



\section{Introduction}

Multiple orthogonal polynomials have historically been closely associated with Hermite–Padé theory and its applications in constructive function theory. For thorough introductions to multiple orthogonal polynomials, refer to the foundational text by Nikishin and Sorokin \cite{nikishin_sorokin} and Van Assche's chapter in \cite[Ch.~23]{Ismail}. An inspiring basic overview can also be found in \cite{andrei_walter}. Research in this field remains highly active: for asymptotics of zeros, see \cite{Aptekarev_Kaliaguine_Lopez}; for the \( LU \)-factorization perspective, \cite{afm}; for Christoffel perturbations, \cite{BFM_22,manas_rojas}; and for applications to random matrix theory, \cite{Bleher_Kuijlaars}.

Mixed multiple orthogonal polynomials \cite{afm,Evi_Arno,sorokin} have applications in stochastic processes, such as Brownian bridges or non-intersecting Brownian motions that start from \( p \) points and end at \( q \) points \cite{Evi_Arno}. They also feature in the study of multicomponent Toda lattices \cite{adler,afm}. Additionally, mixed multiple orthogonal polynomials play a significant role in number theory. Apéry, for instance, demonstrated the irrationality of \( \zeta(3) \) using a mixed Hermite–Padé approximation to three functions \cite{Apery}. These techniques have also been used to prove that infinitely many values of the Riemann zeta function at odd integers are irrational \cite{Ball_Rivoal}, and that at least one of the values \( \zeta(5), \zeta(7), \zeta(9), \zeta(11) \) is irrational \cite{Zudilin}.

Recently, in a series of works \cite{BDFMA,BFM_23}, we explored the application of type I and II multiple orthogonal polynomials to certain Markov chains, specifically those describing non-simple random walks (i.e., beyond the birth and death process). Our work culminated in a spectral Favard theorem with applications to these Markov chains, which are described by bounded banded \((p+2)\)-diagonal oscillatory Hessenberg operators that admit positive bidiagonal factorizations.

The primary result in \cite{BFM3} is that bounded banded Hessenberg matrices with positive bidiagonal factorizations are characterized by a set of positive Stieltjes–Lebesgue measures and can be spectrally described using multiple orthogonal polynomials. This generalizes the spectral Favard theorem for Jacobi matrices to the non-normal case (see\cite{Ismail}). In \cite{BFM2}, we discussed this positive bidiagonal factorization in the context of tetradiagonal Hessenberg matrices. A crucial aspect of the approach in \cite{BFM3} is the multiple Gauss quadrature formula we derived, which provides the exact degrees of precision necessary to describe the spectrality of these operators. We extended these results to mixed multiple orthogonal polynomials in \cite{BFM1}, where we presented a spectral Favard theorem for banded matrices with positive bidiagonal factorizations and included the normality characterization in \cite{BFM1_1}. For a review on these issues see \cite{Contemporary}.

Bidiagonal factorizations are fundamental tools in the study of totally positive matrices \cite{Fallat-Johnson,Pinkus book}, and they are pivotal for establishing a spectral Favard theorem for banded matrices \cite{BFM2,BFM1,BFM1_1}. These factorizations also play a role in Hermite–Padé analysis \cite{Aptekarev_Kaliaguine_VanIseghem}, integrable systems, and Darboux transformations \cite{dolores_ana_ab1}, as well as in the construction of Christoffel transformations \cite{BFM_22}. From the perspective of Markov chains, bidiagonal factorization corresponds to the stochastic factorization of the Markov transition matrix \cite{BDFM_finite}. Bidiagonal factorizations for recurrence matrices of specific families of multiple orthogonal polynomials have been presented in \cite{Aptekarev_Kaliaguine_VanIseghem,BFM_24,BFM_23,Lima-Loureiro}, with further connections to continued branched fractions and combinatorics explored in \cite{Lima,Sokal}.

In this brief paper, we present a comprehensive algorithmic approach to constructing bidiagonal factorizations using the coefficients of the Christoffel transformations of mixed multiple orthogonal polynomials. This method is particularly useful when these transformations map within the same family of polynomials and explicit expressions for the polynomials are available. Our technique is based on the \( LU \)-factorization framework of multiple orthogonality, as discussed in \cite{afm,manas}. It extends to the mixed case the methodology introduced in \cite{BDFHM} though we have omitted, in this work, the discussion on branched continued fractions.


\subsection{The moment matrix}
Let's delve into the scenario of a rectangular matrix of measures:
\begin{align*}
	\d	\mu=\begin{bNiceMatrix}
 \d\mu_{1,1}&\Cdots &\d\mu_{1,p}\\
 \Vdots & & \Vdots\\
 \d	\mu_{q,1}&\Cdots &\d\mu_{q,p}
	\end{bNiceMatrix},
\end{align*}
where the measures \(\mu_{i,j}\) are confined to be supported on the interval \(\Delta \subseteq \mathbb{R}\).
For \(r \in \mathbb{N}\), we consider the matrix of monomials:
\begin{align*}
X_{[r]}(x) = 
\begin{bNiceMatrix}
	I_r \\
	xI_r \\
	x^2 I_r \\
	\Vdots
\end{bNiceMatrix} ,
 \end{align*}
where \( I_r \) is the usual identity matrix of order \( r \).
The moment matrix is defined as:
\begin{align*}
\mathscr{M}(\mu) \coloneqq \int_{\Delta} X_{[q]}(x) \, \mathrm{d}\mu(x) \, X_{[p]}^\top(x).
 \end{align*}
When it's clear from context, for clarity, we will drop the explicit dependence on the measure \(\mu\) and simply write \(\mathscr{M}\).

\subsection{Gauss--Borel factorization}

If all the leading principal submatrices \(\mathscr{M}^{[k]}\) are nonsingular, then the Gauss--Borel factorization exists:
\begin{align*}
\mathscr{M} = \mathscr{L}^{-1} \mathscr{U}^{-1},
 \end{align*}
where \(\mathscr{L}\) is a nonsingular lower triangular semi-infinite matrix and \(\mathscr{U}\) is a nonsingular upper triangular matrix.
 
In this paper, we assume that all leading principal submatrices \(\mathscr{M}^{[k]}\) are nonsingular, ensuring that the moment matrix always admits a Gauss–Borel factorization.

 When necessary, we will denote these triangular matrices as \(\mathscr{L}(\mu)\) and \(\mathscr{U}(\mu)\), indicating the measure \(\mu\) from which they are constructed. It's important to note that this factorization is not unique due to the freedom:
\begin{align*}
\mathscr{L} \to \mathscr{d}^{-1}\mathscr{L}, \quad \mathscr{U} \to \mathscr{U}\mathscr{d},
 \end{align*}
where \(\mathscr{d}\) is any nonsingular diagonal matrix.

Each choice of the invertible diagonal matrix \(\mathscr{d}\) results in a distinct factorization. Two important normalizations are:
\begin{enumerate}[\rm i)]
	\item The left normalization involves setting \(\mathscr{L}\) as a lower unitriangular matrix. When applicable, we will represent the corresponding triangular matrices as \(\mathscr{L}_L\) and \(\mathscr{U}_L\).
	\item The right normalization involves setting \(\mathscr{U}\) as an upper unitriangular matrix. When applicable, we will denote the corresponding triangular matrices as \(\mathscr{L}_R\) and \(\mathscr{U}_R\).
\end{enumerate}

If all leading principal submatrices \(\mathscr{M}^{[k]}\) are nonsingular, the Gauss–Borel factorization is uniquely given by:
\begin{align*}
\mathscr{M} = \mathscr{L}_L^{-1} \mathscr{D} \mathscr{U}_R^{-1}, 
 \end{align*}
in terms of unitriangular matrices and a nonsingular diagonal matrix. 

\subsection{Mixed multiple orthogonal polynomials on the step-line}

Associated with the Gauss--Borel factorization, let's consider the following matrices of polynomials:
\begin{align*}
\begin{aligned}
	B(x) &= \mathscr{L} X_{[q]}(x), & A(x) &= X_{[p]}^\top(x) \mathscr{U}.
\end{aligned}
 \end{align*}
We represent these matrices in terms of their polynomial entries as follows:
\begin{align*}
\begin{aligned}
	B &= \begin{bNiceMatrix}
 B^{(1)}_0 & \Cdots & B^{(q)}_0 \\
 B^{(1)}_1 & \Cdots & B^{(q)}_1 \\
 B^{(1)}_2 & \Cdots & B^{(q)}_2 \\
 \Vdots[shorten-end=-0pt] & & \Vdots[shorten-end=-0pt]
	\end{bNiceMatrix}, &
	A &= \left[\begin{NiceMatrix}
 A^{(1)}_0 & A^{(1)}_1 & A^{(1)}_2 & \Cdots \\
 \Vdots & \Vdots & \Vdots & \\
 A^{(p)}_0 & A^{(p)}_1 & A^{(p)}_2 & \Cdots
	\end{NiceMatrix}\right].
\end{aligned}
 \end{align*}
We have the following relations:
\begin{align*}
\int_{\Delta} B(x) \, \mathrm{d}\mu(x) \, A(x) = I,
 \end{align*}
whose entries are given by the biorthogonality relations:
\begin{align*}
\int_{\Delta} \sum_{b=1}^q \sum_{a=1}^p B^{(b)}_n(x) \, \mathrm{d}\mu_{b,a}(x) \, A^{(a)}_m(x) = \delta_{n,m}.
 \end{align*}
From the Gauss--Borel factorization, it also follows that:
\begin{align*}
	\int_\Delta B(x) \, \mathrm{d}\mu(x) X_{[p]}^\top(x) &= \mathscr{U}^{-1}, \\
	\int_\Delta X_{[p]}(x) \, \mathrm{d}\mu(x) A(x) &= \mathscr{L}^{-1},
\end{align*}
which entrywise represent the following mixed multiple orthogonality relations on the step-line:
\begin{align*}
	\int_\Delta x^l \sum_{a=1}^p \mathrm{d}\mu_{b,a}(x) A_n^{(a)}(x) &= 0, 
	& \begin{aligned}
 & b \in \{1, \ldots, q\}, \quad l \in \left\{0, \ldots, \left\lceil\frac{n-b+1}{q}\right\rceil-1\right\},
	\end{aligned} \\
	\int_\Delta \sum_{b=1}^q B_n^{(b)}(x) \mathrm{d}\mu_{b,a}(x) x^l &= 0, 
	& \begin{aligned}
 & a \in \{1, \ldots, p\}, \quad l \in \left\{0, \ldots, \left\lceil\frac{n-a+1}{p}\right\rceil-1\right\}.
	\end{aligned}
\end{align*}

\subsection{Banded recurrence matrix}

For \(r \in \mathbb{N}_0\), the shift block matrix is:
\begin{align*}
\Lambda_{[r]} \coloneq 
\left[
\begin{NiceMatrix}
	0_r & I_r & 0_r & \Cdots \\
	0_r & 0_r & I_r & \Ddots \\
	0_r & 0_r & 0_r & \Ddots \\
	\Vdots & \Ddots[shorten-end=3pt] & \Ddots[shorten-end=7pt] & \Ddots[shorten-end=9pt]
\end{NiceMatrix}
\right],
 \end{align*}
and for \(r = 1\) we denote \(\Lambda_{[1]}\) as \(\Lambda\). Note that \(\Lambda_{[r]} = \Lambda^r\). These shift matrices have the important property:
\begin{align*}
\Lambda_{[r]}X_{[r]}(x) = x X_{[r]}(x).
 \end{align*}

The moment matrix \(\mathscr{M}\) possesses a Hankel-type symmetry relation that can be expressed as:
\begin{align*}
\Lambda_{[q]}\mathscr{M} = \mathscr{M}\Lambda_{[p]}^\top.
 \end{align*}
From this relation and the Gauss--Borel factorization, we derive:
\begin{align}\label{eq:banded_recurrence_matrix}
	T = \mathscr{L} \Lambda_{[q]} \mathscr{L}^{-1} = \mathscr{U}^{-1} \Lambda_{[p]}^\top \mathscr{U} ,
\end{align}
where
the matrix \(T\) is a \((p,q)\)-banded matrix, with \(p\) subdiagonals and \(q\) superdiagonals. Moreover, the following relations are satisfied
\begin{align*}
T B(x) = x B(x), \quad A(x) T = x A(x),
 \end{align*}
meaning \(B\) and \(A\) are right and left eigenvectors of  \(T\), respectively. These equations represent recurrence relations among the mixed multiple orthogonal polynomials. Therefore, \(T\) is known as the recurrence matrix.

\subsection{Elementary Christoffel perturbations}
Let's introduce the following \(r \times r\) polynomial matrix:
\begin{align*}
\mathfrak{X}_{[r]}(x) = 
\begin{bNiceMatrix}
	0 & 1 & 0 & \Cdots & 0 & 0 \\
	0 & 0 & 1 & \Ddots & & 0 \\
	\Vdots & \Vdots & \Ddots[shorten-end=10pt] & \Ddots & & \Vdots \\
	0 & 0 & & & 1 & 0 \\
	0 & 0 & \Cdots & & 0 & 1 \\
	x & 0 & \Cdots & & 0 & 0
\end{bNiceMatrix}.
 \end{align*}
Notice that \(\mathfrak{X}_{[r]}^r = xI_r\) and that:
\begin{align*}
X_{[r]}\mathfrak{X}_{[r]} = \Lambda X_{[r]}.
 \end{align*}

For \(k \in \mathbb{N}_0\coloneq \{0,1,2,\dots\}\), we consider new matrices of measures given by:
\begin{align}\label{eq:Christoffel}
	\begin{aligned}
	\mathrm{d}\mu_L^{(k)} &\coloneq \mathrm{d}\mu \left(\mathfrak{X}_{[p]}^k\right)^\top, \\
	\mathrm{d}\mu_R^{(k)} &\coloneq \mathfrak{X}_{[q]}^k \mathrm{d}\mu.
\end{aligned}
\end{align}
For \(k=0\), we have \(\mathrm{d}\mu_R^{(0)} = \mathrm{d}\mu_L^{(0)} = \mathrm{d}\mu\).
Note that these new matrices of measures are known as Christoffel perturbations of the matrix of measures \(\mathrm{d}\mu\). For the corresponding perturbed moment matrices, we have:
\begin{align*}
\begin{aligned}
	\mathscr{M}_L^{(k)} &\coloneq \mathscr{M}\left(\mu_L^{(k)}\right) = 
	\mathscr{M}(\mu) \left(\Lambda^k\right)^\top, \\
	\mathscr{M}_R^{(k)} &\coloneq \mathscr{M}\left(\mu_R^{(k)}\right) =
	\Lambda^k \mathscr{M}(\mu).
\end{aligned}
 \end{align*}
For each \(k \in \mathbb{N}_0\), we assume that all leading principal submatrices \(\left(\mathscr{M}^{(k)}\right)^{[l]}\) are nonsingular, guaranteeing the existence of a Gauss–Borel factorization for the sequence of Christoffel transformations.
 
We will denote:
\begin{align*}
\begin{aligned}
	\mathscr{L}_L^{(k)} &\coloneq \mathscr{L}\left(\mu_L^{(k)}\right), & \mathscr{L}_R^{(k)} &\coloneq \mathscr{L}\left(\mu_R^{(k)}\right), \\
	\mathscr{U}_L^{(k)} &\coloneq \mathscr{U}\left(\mu_L^{(k)}\right), & \mathscr{U}_R^{(k)} &\coloneq \mathscr{U}\left(\mu_R^{(k)}\right).
\end{aligned}
 \end{align*}

\section{Bidiagonal factorizations of the banded recurrence matrix}
Let's discuss the \( k\)-th left or right Christoffel perturbations. For each case, we adopt the left or right normalizations. In the left normalization \(\mathscr{L}^{(k)}_L\) are lower unitriangular matrices,
while in right normalization \(\mathscr{U}^{(k)}_R\) are upper unitriangular matrices.

\begin{pro}
\begin{enumerate}[\rm i)]
 \item For the \( k \)-th left Christoffel transformations, there exist lower bidiagonal matrices with unit diagonal \( L_1,\ldots, L_k\) such that
\begin{align}\label{eq:left_Christoffel} 
 \begin{aligned}
 \mathscr{U}_L^{-1}\left(\Lambda^k\right)^\top &= L_1 \cdots L_k \left(\mathscr{U}^{(k)}_L\right)^{-1}, \\
 \mathscr{L}^{(k)}_L &= L_k^{-1} \cdots L_1^{-1} \mathscr{L}_L.
 \end{aligned}
\end{align}
\item 
 For the \( k \)-th right Christoffel transformations, there exist upper bidiagonal matrices with unit diagonal \( U_1, \ldots, U_k\) such that
\begin{align}\label{eq:right_Christoffel}
	\begin{aligned}
 \Lambda^k \mathscr{L}_R^{-1} &= \left(\mathscr{L}_R^{(k)}\right)^{-1} U_k \cdots U_1, \\
 \mathscr{U}_R^{(k)} &= \mathscr{U}_R U_1^{-1} \cdots U_k^{-1}.
	\end{aligned}
\end{align}
	\end{enumerate}
\end{pro}

\begin{proof}
	\begin{enumerate}[\rm i)]
 \item From the Gauss--Borel factorizations of \(\mathscr{M}_L^{(1)}\) and \(\mathscr{M}\), we obtain:
\begin{align*}
\left(\mathscr{L}^{(1)}_L\right)^{-1}\left( \mathscr{U}^{(1)}_L\right)^{-1} = \mathscr{L}_L^{-1} \mathscr{U}_L^{-1}\Lambda^\top.
 \end{align*}
The matrix \(\mathscr{U}^{-1}\Lambda^\top\) is an upper Hessenberg matrix whose Gauss--Borel factorization is:
\begin{align*}
\mathscr{U}_L^{-1}\Lambda^\top = L_1 \left(\mathscr{U}^{(1)}_L\right)^{-1},
 \end{align*}
where \(L_1\) is a lower bidiagonal matrix with unit diagonal. Consequently:
\begin{align*}
\mathscr{L}_L^{(1)} = L_1^{-1} \mathscr{L}_L.
 \end{align*}
In general, since
\begin{align*}
\mathscr{M}_L^{(k+1)} = \mathscr{M}_L^{(k)}\Lambda^\top,
 \end{align*}
we obtain:
\begin{align*}
\begin{aligned}
	\left( \mathscr{U}_L^{(k)} \right)^{-1}\Lambda^\top &= L_{k+1} \left(\mathscr{U}^{(k+1)}_L\right)^{-1}, \\
	\mathscr{L}^{(k+1)}_L &= L_{k+1}^{-1} \mathscr{L}_L^{(k)},
\end{aligned}
 \end{align*}
where \(L_k\) is a lower bidiagonal matrix with unit diagonal. Therefore, \eqref{eq:left_Christoffel} is proven.

\item Let's now discuss the right Christoffel perturbations. In this case, we opt for a right normalization, meaning all the matrices \(\mathscr{U}_R^{(k)}\) are assumed to be upper unitriangular matrices.
For the Gauss--Borel factorization of \(\mathscr{M}_R^{(1)}\), we have:
\begin{align*}
\mathscr{M}_R^{(1)} = \Lambda \mathscr{L}_R^{-1} \mathscr{U}_R^{-1}.
 \end{align*}
The matrix \(\Lambda \mathscr{L}^{-1}\) is a lower Hessenberg matrix whose Gauss--Borel factorization~is:
\begin{align*}
\Lambda \mathscr{L}_R^{-1} = \left(\mathscr{L}_R^{(1)}\right)^{-1} U_1,
 \end{align*}
where \(U_1\) is an upper bidiagonal matrix with unit diagonal. Consequently:
\begin{align*}
\mathscr{U}_R^{(1)} = \mathscr{U}_R U_1^{-1}.
 \end{align*}
In general, since
\begin{align*}
\mathscr{M}_R^{(k+1)} = \Lambda \mathscr{M}_R^{(k)},
 \end{align*}
we obtain:
\begin{align*}
\begin{aligned}
	\Lambda \left( \mathscr{L}_R^{(k)} \right)^{-1} &= \left(\mathscr{L}_R^{(k+1)}\right)^{-1} U_{k+1}, \\
	\mathscr{U}_R^{(k+1)} &= \mathscr{U}_R^{(k)} U_{k+1}^{-1}.
\end{aligned}
 \end{align*}
Consequently \eqref{eq:right_Christoffel} follows.
	\end{enumerate}
\end{proof}

Let's now consider the banded recurrence matrices \(T_L\) and \(T_R\) associated with the left and right normalizations of the Gauss--Borel factorization. We will use the notation:
\begin{align*}
\begin{aligned}
\widetilde{L}_k &\coloneq \mathscr{D}^{-1} L_k \mathscr{D},&	\widetilde{U}_k &\coloneq \mathscr{D} U_k \mathscr{D}^{-1}.
\end{aligned}
 \end{align*}
Note that these new matrices are bidiagonal and with unit diagonal.
\begin{teo}
The following bidiagonal factorizations of the left and right normalizations of the banded recurrence matrix hold:
\begin{align}\label{eq:T_L_bidiagonal}
	T_L& = L_1 \cdots L_p \mathscr{D}^{(p)}\mathscr{D}^{-1} \widetilde{U}_q \cdots \widetilde{U}_1,\\\label{eq:T_R_bidiagonal}
	T_R &= \widetilde{L}_1 \cdots \widetilde{L}_p \mathscr{D}^{-1} \mathscr{D}^{(q)} U_q\cdots U_1.
\end{align}
\end{teo}

\begin{proof}
 From \eqref{eq:banded_recurrence_matrix}, we know that:
\begin{align*}
\begin{aligned}
	T_L &= \mathscr{U}_L^{-1} \left(\Lambda^p\right)^\top \mathscr{U}_L = L_1 \cdots L_p \left(\mathscr{U}^{(p)}_L\right)^{-1}\mathscr{U}_L, \\
	T_R &= \mathscr{L}_R \Lambda^q \mathscr{L}_R^{-1} = \mathscr{L}_R \left(\mathscr{L}_R^{(q)}\right)^{-1}U_q \cdots U_1.
\end{aligned}
 \end{align*}
Recalling that \(\mathscr{U}_L^{(k)} = \mathscr{U}_R^{(k)}\left(\mathscr{D}^{(k)}\right)^{-1}\), we obtain:
\begin{align*}
\begin{aligned}
	\left(\mathscr{U}^{(p)}_L\right)^{-1}\mathscr{U}_L &= \mathscr{D}^{(p)}\left(\mathscr{U}^{(p)}_R\right)^{-1}\mathscr{U}_R \mathscr{D}^{-1} = \mathscr{D}^{(p)}U_q \cdots U_1 \mathscr{D}^{-1}.
\end{aligned}
 \end{align*}
Hence,we find \eqref{eq:T_L_bidiagonal}.

In a similar fashion,
\begin{align*}
\begin{aligned}
	\mathscr{L}_R \left(\mathscr{L}_R^{(q)}\right)^{-1} &= \mathscr{D}^{-1} \mathscr{L}_L\left(\mathscr{L}_L^{(q)}\right)^{-1} \mathscr{D}^{(q)} \\
	&= \mathscr{D}^{-1} L_1\cdots L_p \mathscr{D}^{(q)}.
\end{aligned}
 \end{align*}
Thus, we have proven \eqref{eq:T_R_bidiagonal}.
\end{proof}
From the proof of the previous result we get:
\begin{coro}
	The bidiagonal matrices satisfy
\begin{align*}
\begin{aligned}
	L_{k} &= \left( \mathscr{U}_L^{(k-1)} \right)^{-1}\Lambda^\top\mathscr{U}^{(k)}_L =\mathscr{L}_L^{(k-1)}\left(\mathscr{L}^{(k)}_L\right)^{-1}, \\
	U_k &= \mathscr{L}_R^{(k)}\Lambda \left( \mathscr{L}_R^{(k-1)} \right)^{-1} =\mathscr{U}_R^{(k-1)}\left(\mathscr{U}^{(k)}_R\right)^{-1},
\end{aligned}
 \end{align*}
and the corresponding entries are
\begin{align*}
\begin{aligned}
	(L_k)_{n+1,n}&=\frac{(\mathscr U_L^{(k)})_{n,n}}{(\mathscr U_L^{(k-1)})_{n+1,n+1}}=(\mathscr L_L^{(k-1)})_{n+1,n}-(\mathscr L_L^{(k)})_{n+1,n},\\
 (U_k)_{n,n+1}&=\frac{(\mathscr L_R^{(k)})_{n,n}}{(\mathscr L_R^{(k-1)})_{n+1,n+1}}=(\mathscr U_R^{(k-1)})_{n,n+1}-(\mathscr U_R^{(k)})_{n,n+1}.
\end{aligned} \end{align*}

\end{coro}

Let's introduce \(r(n,p)\coloneq \frac{n}{p}-\big\lfloor\frac{n}{p}\big\rfloor\) as the remainder in the integer division of \(n\) by \(p\). After \(k\) left Christoffel perturbations, let us denote by \(A_L^{(k)}\) the corresponding left-normalized mixed multiple orthogonal polynomials. Similarly, after \(k\) right Christoffel perturbations, let us denote by \(B_R^{(k)}\) the corresponding right-normalized mixed multiple orthogonal polynomials.

Given a polynomial \( P\) we denote its leading coefficient by \( \operatorname{LC}
(P) \).
Then, we can express the bidiagonal matrix entries as follows:
\begin{coro}\label{coro}
	In terms of the leading coefficients of the Christoffel-perturbed mixed multiple orthogonal polynomials, we have:
\begin{align*}
\begin{aligned}
 (L_k)_{n+1,n}&=\frac{\operatorname{LC}\left((A_L^{(k)})^{(r(n,p)+1)}_n\right)}{\operatorname{LC}\left((A_L^{(k-1)})^{(r(n+1,p)+1)}_{n+1}\right)},\\
 (U_k)_{n,n+1}&=\frac{\operatorname{LC}\left((B_R^{(k)})^{(r(n,q)+1)}_n\right)}{\operatorname{LC}\left((B_R^{(k-1)})^{(r(n+1,q)+1)}_{n+1}\right)}.
	\end{aligned} \end{align*}
\end{coro}

For \(k\in\{1,\ldots,p\}\), let's denote \(T_L^{(k)}\) the recurrence matrix for the \(k\)-th iteration of left Christoffel transformations with left normalization. Similarly, for \(k\in\{1,\ldots,q\}\), let's denote \(T_R^{(k)}\) as the recurrence matrix for the \(k\)-th iteration of right Christoffel transformations with right normalization. Next, we'll show that these banded matrices correspond to the Darboux transformations of the original recurrence matrix.

\begin{pro}
	The perturbed recurrence matrices are the Darboux transformations of the recurrence matrix:
	\begin{align*}
 T^{(k)}_L&= L_{k+1}\cdots L_p \mathscr D^{(p)}\mathscr D^{-1} \widetilde U_q\cdots \widetilde U_1L_1\cdots L_k, & k&\in\{1,\ldots,p\},\\
 	T^{(k)}_R&= U_{k}\cdots U_1\widetilde L_{1}\cdots \widetilde L_p \mathscr D^{-1} \mathscr D^{(q)} U_q\cdots U_{k+1}, & k&\in\{1,\ldots,q\}.
	\end{align*}
\end{pro}

\begin{proof}
	The iteration of the left Christoffel perturbations leads to the Darboux transformations of the banded recurrence matrix with left normalization:
\begin{align*}
\begin{aligned}
	T^{(k)}_L &= \mathscr{U}_L^{(k)} \left(\Lambda^p\right)^\top \left(\mathscr{U}_L^{(k)} \right)^{-1} \\
	&= L_k^{-1} \cdots L_1^{-1} \mathscr{U}_L^{-1}\left(\Lambda^p\right)^{\top} (\Lambda^k)^{\top} \left(\mathscr{U}_L^{(k)} \right)^{-1} \\
	&= L_k^{-1} \cdots L_1^{-1} \mathscr{U}_L^{-1}\left(\Lambda^p\right)^{\top} \mathscr{U}_L\mathscr{U}_L^{-1}\left(\Lambda^k\right)^{\top} \left(\mathscr{U}_L^{(k)} \right)^{-1} \\
	&= L_k^{-1} \cdots L_1^{-1} T_L L_1\cdots L_k,
\end{aligned}
 \end{align*}
which for \(k \in \{1,\ldots, p-1\}\) leads to Darboux's permutation of factors in the bidiagonal factorization of \(T_L\).

A similar result holds for the right Christoffel perturbation in the right normalization:
\begin{align*}
\begin{aligned}
	T^{(k)}_R &= U_k \cdots U_1 T_R U_1^{-1}\cdots U_k^{-1}.
\end{aligned}
 \end{align*}
This, for \(k \in \{1,\ldots, q\}\), leads to Darboux's permutation of factors in the bidiagonal factorization of \(T_R\).
\end{proof}

\section{Examples: mixed Jacobi--Piñeiro and Laguerre of the first kind}

Let's consider the vectors 
\begin{align*}
\vec{\alpha} = \begin{bNiceMatrix}
	\alpha_1 & \Cdots & \alpha_p
\end{bNiceMatrix}, \quad 
\vec{\beta} = \begin{bNiceMatrix}
	\beta_1 & \Cdots & \beta_q
\end{bNiceMatrix},
 \end{align*}
and introduce a mixed Jacobi--Piñeiro matrix of measures of the form
\begin{align*}
\d \mu_{\vec{\alpha}, \vec{\beta}} = 
\begin{bNiceMatrix}
	x^{\beta_1} \\ \Vdots \\ x^{\beta_q}
\end{bNiceMatrix}
\begin{bNiceMatrix}
	x^{\alpha_1} & \Cdots & x^{\alpha_p}
\end{bNiceMatrix} 
(1 - x)^\gamma \, \d x,
 \end{align*}
supported on \(\Delta = [0, 1]\), or mixed Laguerre of the first kind matrix of measures
\begin{align*}
\d \mu_{\vec{\alpha}, \vec{\beta}} = 
\begin{bNiceMatrix}
	x^{\beta_1} \\ \Vdots \\ x^{\beta_q}
\end{bNiceMatrix}
\begin{bNiceMatrix}
	x^{\alpha_1} & \Cdots & x^{\alpha_p}
\end{bNiceMatrix} 
\Exp{-x}\d x,
 \end{align*}
supported on \(\Delta = [0, +\infty)\). Here \(\alpha_i, \beta_j, \gamma > -1\) for all \(i \in \{1, \ldots, p\}\) and \(j \in \{1, \ldots, q\}\). We assume that \(\alpha_{i_1} - \alpha_{i_2} \notin \mathbb{Z}\) for \(i_1 \neq i_2\), and \(\beta_{j_1} - \beta_{j_2} \notin \mathbb{Z}\) for \(j_1 \neq j_2\), to ensure distinct exponents. According to \cite{Ulises-Sergio-Judith}, these configurations form a perfect matrices of measures, implying that for any permissible values of \(\vec{\alpha}\) and \(\vec{\beta}\), the corresponding mixed multiple orthogonal polynomials have maximum degrees, i.e. are normal.

Next, consider the action of the basic Christoffel transformation as described in \eqref{eq:Christoffel}. Denote by \(A\) and \(B\) the affine transformations acting on \(\vec{\alpha}\) and \(\vec{\beta}\), respectively:
\begin{align*}
A \vec{\alpha} = \begin{bNiceMatrix}
	\alpha_p + 1 & \alpha_1 & \Cdots & \alpha_{p-1}
\end{bNiceMatrix}, \quad
B \vec{\beta} = \begin{bNiceMatrix}
	\beta_q + 1 & \beta_1 & \Cdots & \beta_{q-1}
\end{bNiceMatrix}.
 \end{align*}
For any allowed vectors \(\vec{\alpha}\) and \(\vec{\beta}\), define new vectors of parameters:
\begin{align*}
\vec{\alpha}^{(k)} \coloneq A^k \vec{\alpha}, \quad
\vec{\beta}^{(k)} \coloneq B^k \vec{\beta},
 \end{align*}
which, by construction, remain within the allowed range of AT parameters. Thus, the 
Christoffel-transformed matrix of measures retains the mixed Jacobi--Piñeiro or Laguerre of the first kind form, respectively:
\begin{align*}
\left( \d \mu_{\vec{\alpha}, \vec{\beta}} \right)_L^{(k)} = \d \mu_{\vec{\alpha}^{(k)}, \vec{\beta}}, \quad
\left( \d \mu_{\vec{\alpha}, \vec{\beta}} \right)_R^{(k)} = \d \mu_{\vec{\alpha}, \vec{\beta}^{(k)}}.
 \end{align*}
Recalling Corollary \ref{coro}, the coefficients in the bidiagonal factorization are given by:
\begin{align*}
\begin{aligned}
	(L_k)_{n+1,n} &= \frac{\operatorname{LC} \left( \left( A_{L, \vec{\alpha}^{(k)}, \vec{\beta}} \right)^{(r(n,p)+1)}_n \right)}{\operatorname{LC} \left( \left( A_{L, \vec{\alpha}^{(k-1)}, \vec{\beta}} \right)^{(r(n+1,p)+1)}_{n+1} \right)}, \\
	(U_k)_{n,n+1} &= \frac{\operatorname{LC} \left( \left( B_{R, \vec{\alpha}, \vec{\beta}^{(k)}} \right)^{(r(n,q)+1)}_n \right)}{\operatorname{LC} \left( \left( B_{R, \vec{\alpha}, \vec{\beta}^{(k-1)}} \right)^{(r(n+1,q)+1)}_{n+1} \right)}.
\end{aligned}
 \end{align*}
Unfortunately, we cannot yet provide more explicit expressions, as the hypergeometric representations for this family are still unavailable. In the non-mixed standard Jacobi–Piñeiro case, these coefficients were explicitly derived in \cite{BFM_24_2}, leading to the closed-form results in \cite{BDFHM}.

\section*{Conclusions and Outlook}

In this paper, we have demonstrated how the \( LU \) factorization framework for mixed multiple orthogonal polynomials along the step-line can be used to determine the coefficients of the bidiagonal factorization of the associated banded recurrence matrix. This approach relies on the application of iterated Christoffel transformations.

For future research, we aim to explore the combinatorial interpretations of these results and to derive explicit expressions for mixed multiple orthogonal polynomials, particularly for the Jacobi--Piñeiro and Laguerre polynomials of the first kind. Additionally, we plan to investigate the corresponding recurrence matrices, their bidiagonal factorizations, and to identify the regions of positivity and total positivity associated with these structures.

\section*{Acknowledgments}

AB acknowledges Centre for Mathematics of the University of Coimbra 
(funded by the Portuguese Government through FCT/MCTES, DOI: 10.54499/UIDB/00324/2020).

JEFD and AF
acknowledges the CIDMA Center for Research and Development in Mathematics and Applications
(University of Aveiro) and the Portuguese Foundation 
\linebreak
for Science and Technology (FCT) for their support within
projects, \linebreak
DOI: 10.54499/UIDB/04106/2020 \& DOI: 10.54499/UI/BD/152576/2022;
additionally, 
\linebreak
JEFD acknowledges the PhD contract UI/BD/152576/2022 from FCT Portugal.

MM acknowledges Spanish ``Agencia Estatal de Investigación'' research project
\linebreak
{[PID2021- 122154NB-I00]},
\emph{Ortogonalidad y Aproximación con Aplicaciones en Machine Learning y Teoría de la Probabilidad}.

\section*{Declarations}

\begin{enumerate}[\rm i)]
	\item \textbf{Conflict of interest:} The authors declare no conflict of interest.
	\item \textbf{Ethical approval:} Not applicable.
	\item \textbf{Contributions:} All the authors have contribute equally.
	\item \textbf{Data availability:} This paper has no associated data.
\end{enumerate}

\end{document}